\documentclass[reqno,oneside,11pt]{amsart}

\usepackage{amsmath, amssymb,epsfig,verbatim}
\usepackage[all,dvips,arc,curve,frame]{xy}
\usepackage{graphics,epic,eepic} 
\usepackage{enumerate}
\usepackage[mathcal]{euscript}
\usepackage{pslatex} 

\theoremstyle{plain}

\newtheorem{thm}{Theorem}

\newtheorem{pro}[thm]{Proposition}
\newtheorem{lem}[thm]{Lemma}

\theoremstyle{definition}

\newtheorem{rem}[thm]{Remark}
\newtheorem*{note}{Note}

\newtheorem{mino}[thm]{Minoration}

\newtheorem{IH}[thm]{Induction Hypothesis}

\def\C{\mathbb{C}}

\def\N{\mathbb{N}}

\def\Af{\mathbb{A}}

\def\A{{\mathcal{A}}}
\def\E{{\mathcal{E}}}

\def\Aut{{\sf{Aut}}}

\def\cpo{{\mathbb{P}^1}}

\def\cpt{{\mathbb{P}^3}}
\def\cpq{{\mathbb{P}^4}}
\def\id{\text{\rm id}}

\def\PSL{{\sf{PSL}}\,}

\def\SL{{\sf{SL}}\,}

\def\SO{{\sf{SO}}\,}
\def\PSO{{\sf{PSO}}\,}
\def\Oq{{\sf{O}_4(\C)}\,}

\def\wght{{\textbf{\rm w}}}

\def\gen{^{\textbf{\rm gen}}}

\newcommand{\jj}{\,{\rm j}\,}

\newcommand{\bdeg}{\deg\,}
\newcommand{\olddeg}{\deg_{\C^4}\,}
\newcommand{\ged}{{\rm ged}\,}

\renewcommand{\hom}[1]{{#1}^\wght}
\renewcommand{\part}[2]{\tfrac{\partial #2}{\partial x_{#1}}}

\newcommand{\Et}{\E^{12}}
\newcommand{\Eb}{\E_{34}}
\newcommand{\Er}{\E^2_4}
\newcommand{\El}{\E^1_3}

\newcommand{\mat}[4]{\left( \begin{smallmatrix} #1\,&#2\\ #3\,&#4 \end{smallmatrix} \right)}
\newcommand{\bmat}[4]{\begin{pmatrix} #1\,&#2\\ #3\,&#4 \end{pmatrix}}

\addtocounter{section}{0}             

\usepackage[colorlinks, linktocpage, 
citecolor = magenta, linkcolor = magenta]{hyperref}

\begin{document}

\title[Tame automorphisms of an affine quadric threefold]
{The tame and the wild automorphisms of an affine quadric threefold}
\date{}
\author{St\'ephane Lamy}
\address{Mathematics Institute\\
        University of Warwick \\
        Coventry \\
        United Kingdom }
\email{s.lamy@warwick.ac.uk}
\address{{\rm Permanent address: } Institut de Math\'ematiques de Toulouse \\ Universit\'e Paul Sabatier \\ 118 route de Narbonne \\ 31062 Toulouse Cedex 9\\ France}
\email{slamy@math.univ-toulouse.fr}
\thanks{The first author was supported by a Marie Curie Intra European Fellowship, and was on leave from Institut Camille Jordan,  Universit\'e Lyon 1, France.}
\author{St\'ephane V\'en\'ereau}
\address{Mathematisches Institut \\ 
Universit\"at Basel \\
Rheinsprung 21 \\
CH-4051 Basel \\
Switzerland}
\email{stephane.venereau@unibas.ch}
\maketitle

\begin{abstract}
 We generalize the notion of a tame automorphism to the context of an affine quadric threefold and we prove that there exist non-tame automorphisms.
\end{abstract}


\section{Introduction}

A landmark result about the automorphism group of the complex affine space $\Af^3$ is the proof by Shestakov and Umirbaev \cite{SU:main} that there exist some wild automorphisms in $\Aut(\Af^3)$, which are defined as automorphisms that cannot be written as a composition of finitely many triangular and affine automorphisms.
Since then some technical aspects of the proof have been substantially simplified and generalized (see \cite{MY, ku:ineq, Ku:main, Ve}); however we feel that we still lack a full understanding of why the proof works. 

In this note we try to gain insights on the problem by transposing the question to another affine threefold, namely the underlying variety of $\SL_2(\C)$.
Note that if $Q^3 \subset \cpq$ is a smooth projective quadric, and $V = Q^3 \smallsetminus H$ is the complement of a hyperplane section, then $V$ is either isomorphic to $\SL_2(\C)$, or to $\Af^3$ (if the hyperplane $H$ was tangent to $Q^3$). 

Another reason to think that $\Aut(\SL_2(\C))$ should be a close analogue to $\Aut(\Af^3)$ comes from the dimension 2.
If we repeat the previous construction for a smooth quadric surface $Q^2 \simeq \cpo \times \cpo \subset \cpt$, we obtain either an affine quadric isomorphic to $\{y^2 - xz = 0 \}$ or the affine plane $\Af^2$.
The automorphism groups of these affine surfaces are well-known \cite{ML, Lam}: both admit presentations as amalgamated products over two factors, and it is not clear how to point out any qualitative difference between both situations.

The story becomes more interesting in dimension 3, and this is the main point of this note: we claim that the group  $\Aut(\SL_2(\C))$, even if still huge, is in some sense more rigid than $\Aut(\Af^3)$.
It is quite straightforward to define a natural notion of elementary automorphism in the new context, hence also a notion of tame automorphism.
By contrast with the situation of $\Af^3$, it is possible to prove that any tame automorphism admits an elementary reduction, the reduction concerning the degree of the automorphism.
In particular, there is no need to adapt the notion of (non-elementary) reductions of type I-IV of Shestakov and Umirbaev, or of ``Shestakov-Umirbaev reduction'' which are their counterparts in the terminology of Kuroda.

As a consequence, we are able to give a self-contained short proof of the existence of wild automorphisms in $\Aut(\SL_2(\C))$.
This might indicate that $\SL_2(\C)$ is a good toy model to test any attempt for an alternative, hopefully more geometric proof of the result of Shestakov and Umirbaev, which would work in positive characteristic.

The paper is organized as follows.

In Section \ref{sec:tame}, we introduce the tame group of $ \SL_2(\C)$ together with some technical definitions, and we state our main result.

Section \ref{sec:para} is devoted to a proof of a version of the ``parachute'' inequality, which was already a key ingredient in the case of $\Af^3$.

Then in Section \ref{sec:proof} we are able to give a short proof of the main result.

As a consequence we can easily produce some wild automorphisms on $\SL_2(\C)$: this is done in Section \ref{sec:wild}.
Note that another natural generalization of $\Af^3$ would be to consider the complement of a smooth quadric in $\cpt$ (since $\Af^3$ is the complement of a plane).
This gives rises to the underlying variety of $\PSL_2(\C)$. 
We indicate at the end of the paper how to adapt our construction to this case.

\begin{note}
After this work was completed, Ivan Arzhantsev and Sergey Ga{\u\i}fullin kindly indicated to us the existence of their work \cite{AG}. 
In their \S 6 a wild automorphism on the 3-dimensional quadric affine cone is produced. 
The example is essentially the same as ours but the techniques involved in the proof are quite different; in particular they do not rely on a generalization of the Shestakov-Umirbaev theory, which is our main point. 
\end{note}


\section{The tame group of $\SL_2(\C)$} \label{sec:tame}


\subsection{Elementary automorphisms}

We work over the field of complex numbers $\C$.

As mentioned in the introduction we find convenient to identify $\Af^4$ with the space of 2 by 2 matrices, and to choose our smooth affine quadric to be given by the determinant $q = x_1x_4 - x_2x_3$:
$$\SL_2(\C) = \left\lbrace \begin{pmatrix} x_1 & x_2 \\ x_3 & x_4 \end{pmatrix};\; x_1x_4-x_2x_3 = 1 \right\rbrace.$$

The group structure on  $\SL_2(\C)$ will be useful to describe some automorphisms of the underlying variety, but is by no mean essential. 

An automorphism $F$ of $\SL_2(\C)$ is given by the restriction of an endomorphism on $\Af^4$  
$$(x_1,\dots,x_4) \mapsto (f_1, \dots, f_4)$$
where $f_i \in \C[x_1,x_2,x_3,x_4]$. 
Note that the $f_i$'s are only defined up to the ideal $(q-1)$, and that we do not assume \textit{a priori} that $(f_1, \dots, f_4)$ define an automorphism of $\Af^4$.
We usually simply write 
$$F = \bigl( \begin{smallmatrix} f_1&f_2\\ f_3&f_4 \end{smallmatrix} \bigr).$$

The composition of two automorphisms $F$ and $G$ is denoted $F\circ G$, and should not be confused with the matrix multiplication we use in the definitions below.
A word of warning: even if the terminology we are about to introduce is inspired by \cite{Ku:main}, we differ from Kuroda in one crucial point: we consider automorphisms of an affine variety, and not of the corresponding algebra. As a consequence, our composition $F\circ G$ would be denoted $G\circ F$ by Kuroda...

Recall that an elementary automorphism in the context of the affine space $\Af^3$ is an automorphism of the form $(x_1,x_2,x_3) \mapsto (x_1 + P(x_2,x_3), x_2,x_3)$, up to permutation of the variables.
A natural generalization in the context of $\SL_2(\C)$ is to consider automorphisms preserving two coordinates in the matrix.
One can obtain such automorphisms by multiplication by a triangular matrix: for instance if $h \in \C[x_1,x_2]$ then we can consider automorphisms of the form
$$
\begin{pmatrix} x_1 & x_2 \\ 
x_3 + x_1h(x_1,x_2) & x_4 + x_2h(x_1,x_2) \end{pmatrix}
= \begin{pmatrix} 1 & 0 \\ 
h(x_1,x_2) & 1  \end{pmatrix}
\begin{pmatrix} x_1 & x_2 \\ x_3 & x_4 \end{pmatrix}\; .
$$

It turns out to be useful to allow some coefficients; so we shall say that 
\begin{equation*}
\Eb = \left\lbrace \begin{pmatrix} x_1/a & x_2/b \\ 
b x_3 + b x_1h(x_1,x_2) & a x_4 + a x_2h(x_1,x_2) \end{pmatrix} ;\;
a,b \in \C^*, h \in \C[x_1,x_2] \right\rbrace 
\end{equation*}
is the group of elementary automorphisms of type $\Eb$.

One can make a similar construction multiplying on the right and/or using an upper triangular matrix. One then obtains three other types of elementary automorphisms:
\begin{align*}
\Et &= \left\lbrace \begin{pmatrix} ax_1 + ax_3h(x_3,x_4) & bx_2+bx_4h(x_3,x_4) \\ 
x_3/b  & x_4/a  \end{pmatrix} ;\;
a,b \in \C^*, h \in \C[x_3,x_4] \right\rbrace 
\\ 
\Er &= \left\lbrace \begin{pmatrix} x_1/a & bx_2+bx_1h(x_1,x_3) \\ 
x_3/b  & ax_4 +ax_3h(x_1,x_3) \end{pmatrix} ;\;
a,b \in \C^*, h \in \C[x_1,x_3] \right\rbrace 
\\ 
\El &= \left\lbrace \begin{pmatrix} ax_1 + ax_2h(x_2,x_4) & x_2/b \\ 
bx_3 +bx_4h(x_2,x_4) & x_4/a \end{pmatrix} ;\;
a,b \in \C^*, h \in \C[x_2,x_4] \right\rbrace \; .
\end{align*}

The union of these four groups is the set (not the group!) of \textbf{elementary automorphisms}, denoted by
$$\E = \Et \cup \Eb \cup \Er \cup \El\; .$$


\subsection{Affine automorphisms}

If a linear endomorphism $F$ of $\Af^4$ induces an automorphism on $\SL_2(\C)$, then by homogeneity of $q$ we see that $F$ preserves all levels of the determinant, which is a non-degenerate quadratic form on $\Af^4$:
$q \circ F = q.$
In particular $F$ is an element of the complex orthogonal group  $\Oq$.
Note that these automorphisms are exactly the ones that extend biregularly to the natural compactification of $ \SL_2(\C) \subset \Af^4$ as a smooth quadric in $\cpq$.
So in this sense $\Oq$ plays the same role as the affine group for $\Af^3$.

It is a classical fact (see \cite[page 274]{FH}) that $\PSO_4(\C)$ is isomorphic to $\PSL_2(\C) \times \PSL_2(\C)$. 
We can explicitely recover this isomorphism in our setting by looking at the action of $\SL_2(\C) \times \SL_2(\C) / (-\id, -\id)$ on $\SL_2(\C)$ by multiplication on both sides:
\begin{multline*}  
\bmat{a}{b}{c}{d} \bmat{x_1}{x_2}{x_3}{x_4} \bmat{a'}{b'}{c'}{d'}
= \bmat{a}{b}{c}{d} \bmat{a'x_1+c'x_2}{b'x_1+d'x_2}{a'x_3+c'x_4}{b'x_3+d'x_4}\\
= \bmat{aa'x_1+ac'x_2+ba'x_3+bc'x_4}{ab'x_1+ad'x_2+bb'x_3+bd'x_4}{ca'x_1+cc'x_2+da'x_3+dc'x_4}{cb'x_1+cd'x_2+db'x_3+dd'x_4}\; .
\end{multline*}
This gives an embedding of $\SO_4(\C)$ into $\Aut(\SL_2(\C))$, and adding the transpose automorphism
$$ \bmat{x_1}{x_3}{x_2}{x_4}$$ we recover the whole complex orthogonal group $\Oq$.


\subsection{Tame and wild automorphisms}
\label{intro:tame}

We define the \textbf{tame subgroup} of $\Aut(\SL_2(\C))$ as the group generated by elementary automorphisms and $\Oq$.
An element of $\Aut(\SL_2(\C))$ is called \textbf{wild} if it is not tame.


Composing an element from $\Eb$ and another from $\El$ we construct the tame automorphism
\begin{equation*}
\begin{pmatrix} x_1 -x_2h(x_2) & x_2 \\ 
x_3 + (x_1-x_4)h(x_2) -x_2h^2(x_2)& x_4 + x_2h(x_2) \end{pmatrix}
= \begin{pmatrix} 1 & 0 \\ 
h(x_2) & 1  \end{pmatrix}
\begin{pmatrix} x_1 & x_2 \\ x_3 & x_4 \end{pmatrix}
\begin{pmatrix} 1 & 0 \\ 
-h(x_2) & 1  \end{pmatrix}\; .
\end{equation*}

This automorphism is the exponential of the locally nilpotent derivation $ h(x_2)\partial$ where
$$\left\lbrace 
\begin{array}{ccl}
\partial x_1 &=& -x_2 \\
\partial x_2 &=& 0 \\
\partial x_3  &=& x_1-x_4 \\
\partial x_4 &=& x_2 
\end{array}
\right.$$

Note that not only $x_2$ but also the trace $x_1+x_4$ is in the kernel of $\partial$. 
In particular, taking the exponential of $(x_1+x_4)\partial$ we obtain the automorphism
 \begin{equation*}
\sigma = 
\begin{pmatrix} x_1 -x_2(x_1+x_4) & x_2 \\ 
x_3 + (x_1-x_4)(x_1+x_4) -x_2(x_1+x_4)^2 & x_4 + x_2(x_1+x_4) \end{pmatrix}\; .
\end{equation*}

As a consequence of our main result stated below we prove in \S \ref{sec:wild} that $\sigma$ is a wild automorphism.


\subsection{A degree on $\C[x_1,x_2,x_3,x_4]$}

For $f \in \C[x_1,x_2,x_3,x_4] \smallsetminus \{0\}$, we define the degree of $f$ as an element of $\N^3$ by taking 
$$\olddeg x_1^ix_2^jx_3^kx_4^l = (i,j,k,l)
\left(
  \begin{smallmatrix}
    1  & 0 & 0 \\ 
    0 & 1 & 0 \\ 
    1 & 0 & 1 \\ 
    0 & 1 & 1
  \end{smallmatrix}
\right)$$
and using the graded lexicographic order on $\N^3$: we first compare the sums of the coefficients and, in case of a tie, apply the lexicographic order. 
So (recall that $q$ is the determinant, defining the affine quadric)
$$(1,1,1) = \olddeg x_1x_4=\olddeg x_2x_3 = \olddeg q.$$ 
By convention $\deg 0 = -\infty$, with $-\infty$ smaller than any element of $\N^3$.
The leading part  of a polynomial $$p=\sum_{(i,j,k,l)}p_{i,j,k,l}x_1^ix_2^jx_3^kx_4^l\in\C[x_1,x_2,x_3,x_4]$$
will be denoted $\hom{p}$, hence 
$$
\hom{p}=\sum_{\olddeg x_1^ix_2^jx_3^kx_4^l=\olddeg p } p_{i,j,k,l}x_1^ix_2^jx_3^kx_4^l\; .
$$
Remark that $\hom{p}$ is not in general a monomial; for instance $\hom{(q-1)}=\hom{q} = q$.
The notation $^\wght$, for weight,  is borrowed from Kuroda, and intends to recall that the leading part is relative to the particular choice of weights $\wght = \left(
  \begin{smallmatrix}
    1  & 0 & 0 \\ 
    0 & 1 & 0 \\ 
    1 & 0 & 1 \\ 
    0 & 1 & 1
  \end{smallmatrix}
\right)$ we made (note that to recover the notation of \cite{Ku:main} one has to transpose this matrix).

\subsection{A degree on $\C[\SL_2(\C)]$}

We are not so much interested by the degree of elements inside $\C[x_1,x_2,x_3,x_4]$ but more in  the quotient by the ideal $(q-1)$ which corresponds to $\C[\SL_2(\C)]$.  To do this, we use a classical trick (see \cite{KML,Z}) starting from the global $\olddeg$: we define the desired degree, simply denoted $\deg$, as follows:
if $\bar{f}$ is the class of a polynomial $f$ in $\C[x_1,x_2,x_3,x_4]/(q-1)$ then we set 
$$
 \bdeg \bar f = \min \{ \olddeg g; \, \bar{f} = \bar{g}\}\; .
$$ 

Remark that if $\bar p=\bar f$ then the following equivalence holds:
$$\deg \bar f=\olddeg p\Leftrightarrow p^{\rm w}\notin (q)\setminus\{0\} $$
and we will call  such a $p$ a {\bf good representative of } $\bar f$. Let us now check that $\deg$ is in turn a degree function i.e. that 
\begin{itemize}
\item $\deg \bar f=-\infty\Leftrightarrow \bar f=0$, 
\item $\deg (\bar f_1+\bar f_2 )\leq\max\{\deg \bar f_1,\deg \bar f_2\}$,
\item $\deg \bar f_1\bar f_2=\deg \bar f_1+\deg \bar f_2$, $\forall \bar f_1,\bar f_2\in \C[x_1,x_2,x_3,x_4]/(q-1)$. 
\end{itemize} 

The first equivalence is easy. For the two other facts, we pick $p_1$ and $p_2$ good representatives of the $\bar f_i$'s. One has $\overline {p_1+p_2}=\bar f_1+\bar f_2$ hence, by definition of $\deg$, one has
$$
\deg (\bar f_1+\bar f_2) \leq \olddeg p_1+p_2\leq\max\{\olddeg p_1,\olddeg p_2\}=\max\{\deg \bar f_1,\deg \bar f_2\}.
$$ 

As for the third equality, it suffices to prove that $p_1p_2$ is a good representative of $\bar f_1\bar f_2$ i.e. that $\hom{(p_1p_2)}\notin   (q)$. This is the case since $\hom{p_1},\hom{p_2}\notin (q)$ and $(q)$ is a prime ideal.\\

 We also need to define the leading part of an element of $\C[\SL_2(\C)]$. By abuse of notation, we still denote this by $\empty^\wght$, and define it as follows:
$$
  {\bar f}^\wght=p^\wght + (q)\in \C[x_1,x_2,x_3,x_4]/(q)\mbox{ where $p$ is a good representative of }\bar f\; .
$$
Remark that, in contrast with $\olddeg$ on $\C[x_1,x_2,x_3,x_4]$, the elements $\bar f$ and $\hom{\bar f}$ do not belong to the same set anymore: 
$$\bar f\in\C[x_1,x_2,x_3,x_4]/(q-1) \text{ whereas } \hom{\bar f}\in\C[x_1,x_2,x_3,x_4]/(q).$$
One has to check that the definition is independent of the choice of the good representative. Let us take  two good representatives $p_1$, $p_2$ of the same $\bar f\in\C[\SL_2(\C)]$, then $p_2-p_1\in(q-1)$ whence $(p_2-p_1)^\wght\in (q)$. But $(p_2-p_1)^\wght=p_2^\wght$, $-p_1^\wght$ or $p_2^\wght-p_1^\wght$  and since $p_1^\wght,p_2^\wght\notin(q)$ only the last one is possible, thereby giving: $p_2^\wght-p_1^\wght\in (q)$. 
 
From now on, we drop the bars  and work directly with regular functions on $\SL_2(\C)$. 
So for example, $x_1,x_2,x_3,x_4$ should be understood as their restrictions to $\SL_2(\C)$.


\subsection{Elementary reductions and main result}

If $f_1, \dots, f_4$ are elements in $\C[\SL_2(\C)]$ such that $F = \bigl( \begin{smallmatrix} f_1&f_2\\ f_3&f_4 \end{smallmatrix} \bigr) \in \Aut(\SL_2(\C))$, we define $\bdeg F = \sum \bdeg f_i$.

We say that $E\circ F$ is an \textbf{elementary reduction} of $F$ if $E \in \E$ and $\bdeg E\circ F < \bdeg F$. 

We denote by $\A$ the set of tame automorphisms that admit a sequence of elementary reductions to an element of $\Oq$.

The main result of this note is then:

\begin{thm} 
\label{thm:main}
Any tame automorphism of $\SL_2(\C)$ is an element of $\A$.
\end{thm}


\section{The parachute}\label{sec:para}

In this section we shall obtain a minoration for the degree of a polynomial in two algebraically independent regular functions $f_1, f_2\in\C[\SL_2(\C)]$. 
For this, we adapt the techniques used in \cite{Ve} (see also \cite{SU:ineq,ku:ineq}) where the $f_i$'s were in $\C[x_1,\dots,x_n]$.

\subsection{Generic degree}

Given $f_1, f_2\in\C[\SL_2(\C)] \smallsetminus \{0\}$, consider $R = \sum R_{i,j} X_1^iX_2^j\in\C[X_1,X_2]$ a non-zero polynomial in two variables. 
Generically (on the coefficients $R_{i,j}$ of $R$), $\deg R(f_1,f_2)$ coincides with $\ged R$ where $\ged$ (standing for \textbf{generic degree}) is the weighted degree on  $\C[X_1,X_2]$ defined by 
$$\ged X_i =\deg f_i\in\N^3,$$
again with the graded lexicographic order. 
Namely we have
$$ R(f_1,f_2)  =  R\gen(f_1,f_2)+LDT(f_1,f_2)$$
where
$$  R\gen(f_1,f_2) = \sum_{\ged X_1^iX_2^j=\ged R}R_{i,j}f_1^if_2^j$$
is the leading part of $R$ with respect to the generic degree and $LDT$ represents the Lower (generic) Degree Terms. 
One has 
$$\deg LDT(f_1,f_2)<\deg R\gen(f_1,f_2)=\ged R =\deg R(f_1,f_2)$$ 
unless $R\gen(\hom{f_1},\hom{f_2})=0$, in which case the degree falls: $\deg R(f_1,f_2)<\ged R$. 

Let us focus on the condition $R\gen(\hom{f_1},\hom{f_2})=0$. Of course this can happen only if $\hom{f_1}$ and $\hom{f_2}$ are algebraically dependent. Remark that the ideal 
$$I=\{S\in\C[X_1,X_2];\; S(\hom{f_1},\hom{f_2})=0\}$$
must then be principal, prime and generated by a $\ged$-homogeneous polynomial. 
The only possibility is that $I=(X_1^{s_1}-\lambda X_2^{s_2})$ where $\lambda\in\C^*$, $s_1\deg f_1=s_2\deg f_2$ and $s_1,s_2$ are coprime. To sum up, in the case where $\hom{f_1}$ and $\hom{f_2}$ are algebraically dependent one has
\begin{equation}
\deg R(f_1,f_2)<\ged R \;\Leftrightarrow\;  R\gen(\hom{f_1},\hom{f_2})=0  \;\Leftrightarrow\; R\gen\in (H) \, \label{rel}
\end{equation}
where $H=X_1^{s_1}-\lambda X_2^{s_2}$.

\subsection{Pseudo-Jacobians}

If $f_1,f_2,f_3,f_4$ are polynomials in $\C[x_1,x_2,x_3,x_4]$, we denote by  $\jj_{\C^4}(f_1,f_2,f_3,f_4)$ the Jacobian determinant, i.e. the determinant of the Jacobian $4 \times 4$- matrix $(\frac{\partial f_i}{\partial x_j})$.
Then we define the \textbf{pseudo-Jacobian} of $f_1,f_2,f_3$ by the formula
$$\jj(f_1,f_2,f_3):=\jj_{\C^4}(q,f_1,f_2,f_3).$$

\begin{lem}
If $f_i, h_i \in \C[x_1,x_2,x_3,x_4]$ and $g_i = f_i + (q-1)h_i$, for $i = 1, \dots, 4$, then the pseudo-Jacobians $\jj(g_1,g_2,g_3)$ and $\jj(f_1,f_2,f_3)$ are equal up to an element in the ideal $(q-1)$.
In particular, if $f_1, f_2, f_3 \in \C[\SL_2(\C)]$, the pseudo-Jacobian $\jj(f_1,f_2,f_3)$ is a well-defined element of $\C[\SL_2(\C)]$. 
\end{lem}

\begin{proof}
This is an easy consequence of the following two observations: 
\begin{itemize}
\item The Jacobian $\jj_{\C^4}$ and, consequently, the pseudo-Jacobian $\jj$ as well, are $\C$-derivations in each of their entries;
\item One has $0=\jj_{\C^4}(q,q,f_1,f_2)=\jj(q,f_1,f_2)= -\jj(f_1,q,f_2)= \jj(f_1,f_2,q)$. 
\end{itemize}
\end{proof}

\begin{lem} \label{lem:pseudojac}
Assume $f_1, f_2, f_3 \in \C[\SL_2(\C)]$. Then
$$\deg\jj(f_1,f_2,f_3) \leq  \deg f_1+\deg f_2+\deg f_3-(1,1,1).$$
\end{lem}

\begin{proof}
An easy computation shows the following inequality:

$$
\olddeg \jj_{\C^4}(f_1,f_2,f_3,f_4)  \leq \sum_i \olddeg f_i- \sum_{i}\olddeg x_i=
\sum_i \olddeg f_i -(2,2,2)\; .
$$

Recalling the definitions of $\jj$ and $\deg$ we obtain:
\begin{multline*}
\deg\jj(f_1,f_2,f_3) \leq \olddeg \jj_{\C^4} (q,f_1,f_2,f_3) \\ 
\leq \olddeg q+ \sum_i \olddeg f_i -(2,2,2) = \sum_i \olddeg f_i -(1,1,1).
\end{multline*}
Assuming $f_i$ to be a good representative of $f_i+(q-1)$, that is $\deg f_i = \olddeg f_i$ for $i=1,2,3$,  we get, $\forall f_1,f_2,f_3\in\C[\SL_2(\C)]$,
$$\deg\jj(f_1,f_2,f_3)\leq \deg f_1+\deg f_2+\deg f_3-(1,1,1).$$
\end{proof}

We shall essentially use those pseudo-Jacobians with $f_1=x_1,x_2,x_3$ or $x_4$.
Therefore we introduce the notation 
$\jj_k(\cdot,\cdot):=\jj(x_k,\cdot,\cdot)$ for all $k=1,2,3,4$. The inequality from Lemma \ref{lem:pseudojac} gives
$$ \deg\jj_k(f_1,f_2) \leq \deg f_1+\deg f_2+\deg x_k-(1,1,1)$$
from which we deduce
\begin{equation} \label{ineqjb}
\deg\jj_k(f_1,f_2) <  \deg f_1+\deg f_2,\;\forall k=1,2,3,4.
\end{equation}

We shall also need the following observation. 

\begin{lem}\label{lem:pastousnul}
If $f_1,f_2$ are algebraically independent functions in $\C[\SL_2(\C)]$, then the $\jj_k(f_1,f_2)$, $k =1, \dots, 4$, are not simultaneously zero.

In particular $\max_{k=1,2,3,4}\deg\jj_k(f_1,f_2)\neq -\infty$ or, equivalently, $$\max_{k=1,2,3,4}\deg\jj_k(f_1,f_2)\in\N^3.$$
\end{lem}

\begin{proof}
Assume that  $\jj(x_k,f_1,f_2)=0$ $\forall k=1,2,3,4$. Then the four derivations $\partial_k:=\jj(x_k,f_1,\cdot)$ have both $f_1$ and $f_2$ in their kernel. We now need the following well-known   relation between the transcendence degree and the dimension of a derivations (see e.g. \cite[VIII,\S 5]{L} ): if $K\subset L$ is a characteristic 0 field extension then
\begin{eqnarray}\label{trdeg=dim}
  {\rm trdeg}_KL &= & {\rm dim}_L{\rm Der}_KL\;.
\end{eqnarray}
Applied to $K=\C(f_1,f_2)\subset L=\C(\SL_2(\C))$ this gives that any two $\C(f_1,f_2)$-derivations are proportional so $\forall k\neq l\in\{1,2,3,4\}$, any two $\partial_k$, $\partial_l$ are non-trivially related: $a\partial_k+b\partial_l=0$. Evaluating this equality in $x_k$ and $x_l$ gives that $\partial_k(x_l)=\jj(x_k,f_1,x_l)= \partial_l(x_k)=0$ for all $k,l\in\{1,2,3,4\}$. It follows that all the derivations $\partial_{kl}:=\jj(x_k,x_l,\cdot)$ are $\C(f_1)$-derivations and, applying (\ref{trdeg=dim}) again with $K=\C(f_1)$ it follows that any such three $\partial_{kl}$ are related e.g. $a\partial_{12}+b\partial_{13}+c\partial_{23}=0$ with $a,b,c$ not all zero. Again evaluating it on $x_1$, $x_2$, $x_3$ gives $\jj(x_1,x_2,x_3)=0$ (and the same holds for any triple in $\{1,2,3,4\}$). This means that $\jj_{\C^4}(q,x_1,x_2,x_3)=-\frac{\partial q}{\partial x_4}=-x_1$ (with the $x_i$'s regarded as elements of $\C[x_1,x_2,x_3,x_4]$) is zero, a contradiction.
\end{proof} 

\subsection{The parachute}

In this section $f_1, f_2 \in \C[\SL_2(\C)]$ are algebraically independent, and we denote by $d_i \in \N^3$ the degree $\deg f_i$.
We define the \textbf{parachute} of $f_1, f_2$ to be  
$$ \nabla(f_1,f_2) = d_1 + d_2 - \max_{k = 1,2,3,4} \deg\jj_k(f_1,f_2).$$ 
By Lemma \ref{lem:pastousnul}, we immediately remark that $ \nabla(f_1,f_2)\leq d_1+d_2 $.

\begin{lem}\label{lem:ineq1}
Assume $\deg\frac{\partial^n R}{\partial X_2^n}(f_1,f_2)$ coincides with the generic degree $\ged \frac{\partial^n R}{\partial X_2^n}$. 
Then 
$$d_2 \cdot\deg_{X_2}R-n\nabla(f_1,f_2)  < \deg R(f_1,f_2).$$
\end{lem}

\begin{proof}

As already remarked $\jj_{\C^4}$, $\jj$ and now $\jj_k$ as well are $\C$-derivations in each of their entries. We may then apply the chain rule on $\jj_k(f_1,\cdot)$ evaluated in $R(f_1,f_2)$:
$$ \frac{\partial R}{\partial X_2}(f_1,f_2)\jj_k(f_1,f_2)
=  \jj_k(f_1,R(f_1,f_2)).$$
Now taking the degree and applying inequality (\ref{ineqjb}) (with $R(f_1,f_2)$ instead of $f_2$), we obtain
$$\deg\frac{\partial R}{\partial X_2}(f_1,f_2)+\deg\jj_k(f_1,f_2)  <  d_1+\deg R(f_1,f_2).$$
We deduce
$$\deg\frac{\partial R}{\partial X_2}(f_1,f_2)+d_2-(\underbrace{d_1+d_2-\max_{k=1,2,3,4}\deg\jj_k(f_1,f_2)}_{=\nabla(f_1,f_2)})  < \deg R(f_1,f_2).$$
By induction, for any $n \ge 1$ we have
\begin{align*}
\deg\frac{\partial^n R}{\partial X_2^n}(f_1,f_2)+nd_2-n\nabla(f_1,f_2) & < \deg R(f_1,f_2) \; .
\end{align*}
Now if the integer $n$ is as given in the statement one gets:
$$
\deg\frac{\partial^n R}{\partial X_2^n}(f_1,f_2)=\ged \frac{\partial^n R}{\partial X_2^n}\geq d_2\cdot\deg_{X_2}\frac{\partial^n R}{\partial X_2^n} =d_2\cdot(\deg_{X_2}R-n)=d_2\cdot\deg_{X_2}R-d_2n
$$ 
which, together with the previous inequality, gives the result.
\end{proof}

\begin{lem}\label{lem:goodn}
Let $H$ be the generating relation between $\hom{f_1}$ and $\hom{f_2}$ as in the equivalence (\ref{rel}) and $n\in\N$ such that $R\gen\in(H^n)\setminus(H^{n+1})$. Then $n$ fulfills the assumption of Lemma \ref{lem:ineq1} i.e.
$$\deg\frac{\partial^n R}{\partial X_2^n}(f_1,f_2)=\ged \frac{\partial^n R}{\partial X_2^n}\; .$$
\end{lem}

\begin{proof}
It suffices to remark that ${(\frac{\partial R}{\partial X_2})}\gen=\frac{\partial R\gen}{\partial X_2}$ and that $R\gen\in(H^n)\smallsetminus(H^{n+1})$ implies $\frac{\partial R\gen}{\partial X_2}\in(H^{n-1})\smallsetminus(H^{n})$. One concludes by induction.
\end{proof}

Remark that, by definition of $n$ in Lemma \ref{lem:goodn} above, one has $\deg_{X_2} R\geq\deg_{X_2}R\gen\geq ns_2$ which together with Lemma \ref{lem:ineq1} gives (recall that $s_1d_1=s_2d_2$)
\begin{equation}\label{last}
d_1ns_1 -n\nabla(f_1,f_2)  < \deg R(f_1,f_2) 
\end{equation}

\subsection{The minoration}

Now we come to the main result of this section, which is a close analogue of \cite[Lemma 3.3(i)]{Ku:main}.

\begin{mino}\label{mino:para}
Let $f_1,f_2\in\C[\SL_2(\C)]$ be algebraically independent and $R(f_1,f_2)\in\C[f_1,f_2]$.
Assume $R(f_1,f_2) \not\in\C[f_2]$ and $\hom{f_1}\not\in\C[\hom{f_2}]$.
Then
$$
  \deg (f_2R(f_1,f_2))>\deg f_1\, .
$$
\end{mino}

\begin{proof}
Let $n$ be as in Lemma \ref{lem:goodn}. 
If $n=0$ then $\deg R(f_1,f_2)=\ged R\geq \deg f_1$ by the assumption  $R(f_1,f_2) \not\in\C[f_2]$ and then $\deg (f_2R(f_1,f_2))\geq\deg f_2+\deg f_1>\deg f_1$ as wanted. 

If $n\geq 1$ then, by (\ref{last}),
$$ d_1s_1-\nabla(f_1,f_2) < \deg R(f_1,f_2) $$
and, since $\nabla(f_1,f_2) \leq d_1 + d_2$, 
$$  d_1s_1-d_1-d_2 < \deg R(f_1,f_2).$$
We obtain
$$d_1(s_1-1) <  \deg R(f_1,f_2)+d_2=\deg(f_2R(f_1,f_2)).$$
The assumption $\hom{f_1}\not\in\C[\hom{f_2}]$ forbids $s_1$ to be equal to one, hence we get the desired minoration.
\end{proof}


\section{Proof of the main result} \label{sec:proof}

In this section we prove the following proposition, which immediately implies  Theorem \ref{thm:main}.
 
\begin{pro}\label{pro:main}
If $F \in \A$, and $E \in \E$, then $E\circ F \in \A$. 
\end{pro}

The proposition is clear when $\bdeg E\circ F > \bdeg F$.
From now on we assume that  
$$\bdeg E\circ F \le \bdeg F.$$

The result is also clear if $F \in \Oq$. 
We assume the following induction hypothesis:

\begin{equation*}
\textit{If } H \in \A, E \in \E, \textit{ and } \bdeg H < \bdeg F, 
\textit{ then } E\circ H \in \A.
\end{equation*}
For practical reasons we introduce the notation:
$$
  \A_{<F}:=\{H\in\A\vert\deg H<\deg F\}
$$
and rewrite the induction hypothesis in an equivalent formulation:
\begin{IH}\nonumber\label{IH}
If $E\circ G\in\A_{<F}$ for some $E\in\E$ then $G\in\A$.
\end{IH}

We shall use the following basic observation repetitively.

\begin{lem} \label{lem:inA}
Let $F = \mat{f_1}{f_2}{f_3}{f_4} \in \Aut(\SL_2(\C))$, and $a,b \in \C^*$. Then $F \in \A$ if and only if the following equivalent conditions hold:

\begin{enumerate}[$(i)$]
\item   $\mat{af_1}{f_2/b}{bf_3}{f_4/a} \in \A$;
\item $\mat{f_1}{f_2}{f_3}{f_4}\mat{a}{0}{0}{\frac{1}{a}} \in \A$;
\item $\mat{f_1}{f_2}{f_3}{f_4}\mat{0}{a}{-\frac{1}{a}}{0} \in \A$.
\end{enumerate}
\end{lem}

\begin{proof}
$F\in\A\Leftrightarrow(i)$: It is sufficient to prove one implication. 
Assume $F \in \A$.
Then $F$ admits an elementary reduction $E\circ F$, say with $E \in \El$:
$$E \circ F = \bmat{cf_1+cf_2P(f_2,f_4)}{f_2/d}{df_3+df_4P(f_2,f_4)}{f_4/c} \in \A_{<F}.$$

Let $R(x_2,x_4) = abP(bx_2,ax_4)$, and define 
$E' = \bmat{\frac{c}{a}x_1+\frac{c}{a}x_2R(x_2,x_4)}{\frac{b}{d}x_2}{\frac{d}{b}x_1+\frac{d}{b}x_2R(x_2,x_4)}{\frac{a}{c}x_4}.$

Then $E' \circ \mat{af_1}{f_2/b}{bf_3}{f_4/a} = E \circ F$, hence the result.\\

$F\in\A\Leftrightarrow(ii)$: This is just a special case of the previous equivalence with $a=b$.\\

$F\in\A\Leftrightarrow(iii)$: Using the equality 
$$
\mat{a}{0}{0}{\frac{1}{a}} \mat{0}{1}{-1}{0}=\mat{0}{a}{-\frac{1}{a}}{0}
$$ 
and $(ii)$ it suffices to restrict to the case $a=1$. If we denote $R\in\Aut(\SL_2(\C))$ the right-multiplication by $\mat{0}{1}{-1}{0}$ one easily checks that $R\circ \E\circ  R^{-1}=\E$ (actually conjugation under $R$ only exchanges $\El$ and $\Er$) and that $R$ does not affect the degree. It follows that sequences of elementary reductions of $F$ and of $R\circ F=\mat{f_1}{f_2}{f_3}{f_4}\mat{0}{1}{-1}{0}$ are in one to one correspondence.
 \end{proof}

We are back now with the setting of Proposition \ref{pro:main}.
Since $F \in \A$, there exists an elementary automorphism $E'$ such that $\bdeg E' \circ F < \bdeg F$ and $E' \circ F \in \A$ i.e. $E' \circ F \in \A_{<F}$.
Up to conjugacy, and using Lemma \ref{lem:inA}(i), we can assume 
$$E' = \mat{x_1+x_2P(x_2,x_4)}{x_2}{x_3+x_4P(x_2,x_4)}{x_4} \in \El.$$

We distinguish three cases according to the form of the automorphism $E$ in the proposition:
$E \in \El$, $E \in \Er$ or $E \in \Et$ (the case $E \in \Eb$ is equivalent to the latter one, up to conjugacy).

\subsection{Case $E \in \El$} \label{case:easy}

We have
$$E'\circ E^{-1}\circ(E\circ F)= E' \circ F\in\A_{<F}.$$ 
Since $E'\circ E^{-1}\in\El \subset \E$ we can use Induction Hypothesis  \ref{IH} to conclude.

\begin{rem}
This case is extremely simple, but in the following cases it will be convenient to use commutative diagrams, such as the one in Figure \ref{fig:easy}, to visualize the argument.
The vertices of the diagram correspond to tame automorphisms, and the arrows are either composition (on the left) by one elementary automorphism or a change of coefficients allowed by Lemma \ref{lem:inA} below. 
We distinguish automorphisms which are proven in the text to be in $\A_{<F}$, and Induction Hypothesis \ref{IH} means that for any arrow pointing on such an automorphism, the initial automorphism is in $\A$. 
\end{rem}

\begin{figure}[h]
$$\xymatrix{  
 F = \mat{f_1}{f_2}{f_3}{f_4} \ar[d]_<(.3){E'} \ar[r]^>(.8)E
 & E\circ F \ar[dl]^<(.2){E'\circ E^{-1}}\\
 E'\circ F \in \A_{<f}
}$$
\caption{Case $E\in\E^1_3$ }\label{fig:easy}
\end{figure}

For the next two more substantial cases we shall need the following lemma.

\begin{lem}
\label{lem:basic}
Let $F = \mat{f_1}{f_2}{f_3}{f_4} \in \Aut(\SL_2(\C))$.
If $E \in \El$ and $E\circ F =  \mat{f'_1}{f_2}{f'_3}{f_4}$, then
$$\bdeg E \circ F  \sphericalangle \bdeg F \Longleftrightarrow \bdeg f'_1 \sphericalangle \bdeg f_1  \Longleftrightarrow   \bdeg f'_3 \sphericalangle \bdeg f_3$$
for any relation $\sphericalangle\in\{<,>,\le,\ge,=\}$.
\end{lem}

\begin{proof} One has $f_1f_4-f_2f_3=1$ and the $f_i$'s are not constant hence the leading parts must cancel one another: 
$\hom{f_1}\hom{f_4}-\hom{f_2}\hom{f_3}=0$. It follows: $\deg f_1+\deg f_4=\deg f_2+\deg f_3$ and then 
$$\frac{1}{2}\deg F   =  \deg f_1+\deg f_4   =  \deg f_2+\deg f_3 .$$
Similarly
$$\frac{1}{2}\deg E\circ F  =  \deg f'_1+\deg f_4  =  \deg f_2+\deg f_3',$$
hence
$$\frac{1}{2}(\deg E\circ F-\deg F)  =  \deg f_1'-\deg f_1  =  \deg f_3'-\deg f_3 $$
which gives the desired equivalences.
\end{proof}

\subsection{Case $E \in \Er$}

Using Lemma \ref{lem:inA}(i), we can assume that $E = \mat{x_1}{x_2+x_1Q(x_1,x_3)}{x_3}{x_4+x_3Q(x_1,x_3)}$ so we have
$$
E'\circ F=\mat{f_1+f_2P(f_2,f_4)}{f_2}{f_3+f_4P(f_2,f_4)}{f_4}\mbox{ and }E\circ F=\mat{f_1}{f_2+f_1Q(f_1,f_3)}{f_3}{f_4+f_3Q(f_1,f_3)}\; .
$$

If $P(f_2,f_4)$ is not constant, the inequality $\bdeg E' \circ F < \bdeg F$ implies $\bdeg f_1 > \bdeg f_2$ and $\bdeg f_3 > \bdeg f_4$.
But then $\bdeg E \circ F > \bdeg F$, a contradiction. 

Hence $P(f_2,f_4) = p$ is a constant, and $\bdeg f_1 = \bdeg f_2$, $\bdeg f_3 = \bdeg f_4$.
This in turn implies $Q(x_1,x_3) = q$ is a constant.\\

If $pq \neq 1$, we define $r = \frac{q}{1-pq}$, $s = -\frac{p}{1+rp}$ and we compute
\begin{equation}\label{eq:prs}
\begin{split}
\bmat{1}{0}{p}{1} \bmat{1}{r}{0}{1} \bmat{1}{0}{s}{1} 
&= \bmat{\frac{1}{1+rp}}{r}{0}{1+rp} \\
&= \bmat{1}{q}{0}{1} \bmat{\frac{1}{1+rp}}{0}{0}{1+rp} 
\end{split}
\end{equation}

By assumption, $\mat{f_1}{f_2}{f_3}{f_4} \mat{1}{0}{p}{1} \in \A_{<F}$. 
By Induction Hypothesis \ref{IH}, and since, here,  the multiplication by $\mat{1}{r}{0}{1}$ does not change the degree (the second column is added a scalar multiple of the first one which has a strictly smaller degree), we have   
$$\bmat{f_1}{f_2}{f_3}{f_4} \bmat{1}{0}{p}{1} \bmat{1}{r}{0}{1} \in \A_{<F}.$$
Using the Induction Hypothesis again, we get  $\mat{f_1}{f_2}{f_3}{f_4} \mat{1}{0}{p}{1} \mat{1}{r}{0}{1} \mat{1}{0}{s}{1} \in \A$. 
Thus by (\ref{eq:prs}) we have  
$\mat{f_1}{f_2}{f_3}{f_4} \mat{1}{q}{0}{1} \mat{\frac{1}{1+rp}}{0}{0}{1+rp} \in \A$, and using Lemma \ref{lem:inA}$(ii)$ we obtain (see Figure~\ref{fig:pqnot1})
$$ \bmat{f_1}{f_2}{f_3}{f_4} \bmat{1}{q}{0}{1} = E \circ F \in \A.$$

\begin{figure}[h]
$$\xymatrix{  
 F = \mat{f_1}{f_2}{f_3}{f_4} 
\ar[d]_{ \cdot \mat{1}{0}{p}{1} } \ar[rr]^{ \cdot \mat{1}{q}{0}{1} }
&& E \circ F 
\ar[rr]^{ \cdot \mat{\frac{1}{1+rp}}{0}{0}{1+rp} } 
&& \in \A  \\
\in \A_{<F} \ar[rrrr]_{ \cdot \mat{1}{r}{0}{1} }
&&&& \in \A_{<f} \ar[u]_{ \cdot \mat{1}{0}{s}{1} }
}$$
\caption{Case $E\in\Er$, $pq \neq 1$}\label{fig:pqnot1}
\end{figure}

If $pq = 1$, we write 
$$\bmat{1}{q}{0}{1} = \bmat{1}{0}{p}{1} \bmat{1}{-1/p}{0}{1}  \bmat{0}{1/p}{-p}{0}.$$
By Induction Hypothesis \ref{IH} we have $\bmat{f_1}{f_2}{f_3}{f_4} \bmat{1}{0}{p}{1} \bmat{1}{-1/p}{0}{1} \in \A$, and using Lemma \ref{lem:inA}$(iii)$ we obtain (see Figure \ref{fig:pq=1})
$$ \bmat{f_1}{f_2}{f_3}{f_4} \bmat{1}{q}{0}{1} = E \circ F \in \A\; .$$

\begin{figure}[h]
$$\xymatrix{  
 F = \mat{f_1}{f_2}{f_3}{f_4} 
\ar[drr]_{ \cdot \mat{1}{0}{p}{1} } \ar[rr]^{ \cdot \mat{1}{q}{0}{1} }
&& E \circ F  
&& \in \A \ar[ll]_{ \cdot \mat{0}{1/p}{-p}{0} } \\
&&\in \A_{<F} \ar[urr]_{ \cdot \mat{1}{-1/p}{0}{1} }
}$$
\caption{Case $E\in\Er$, $pq = 1$}\label{fig:pq=1}
\end{figure}

\subsection{Case $E \in \Et$}

Using Lemma \ref{lem:inA}$(i)$, we can assume that $ E = \mat{x_1+x_3Q(x_3,x_4)}{x_2+x_4Q(x_3,x_4)}{x_3}{x_4}$
so we have
$$
E'\circ F=\mat{f_1+f_2P(f_2,f_4)}{f_2}{f_3+f_4P(f_2,f_4)}{f_4}\mbox{ and }E\circ F=\mat{f_1+f_3Q(f_3,f_4)}{f_2+f_4Q(f_3,f_4)}{f_3}{f_4} \; .
$$

Assume first that Minoration \ref{mino:para} is applicable to both $P(f_2,f_4)$ and $Q(f_3,f_4)$. We obtain the contradictory sequence of inequalities:
\begin{align*}
\bdeg f_2 &< \bdeg (f_4P(f_2,f_4)) && (\mbox{Minoration \ref{mino:para} applied to } P), \\
\bdeg (f_4P(f_2,f_4)) &= \bdeg f_3 &&  (\bdeg E' \circ F < \bdeg F), \\
\bdeg f_3 &< \bdeg (f_4Q(f_3,f_4)) && (\mbox{Minoration \ref{mino:para} applied to } Q), \\
\bdeg (f_4Q(f_3,f_4)) &\le \bdeg f_2 && (\bdeg E \circ F \le \bdeg F). 
\end{align*}

We conclude with the following lemma.

\begin{lem}
\label{lem:conditions}
If Minoration \ref{mino:para} is not applicable to either $P(f_2,f_4)$ or $Q(f_3,f_4)$, i.e. if 
$$Q(f_3,f_4) \in \C[f_4],\; \hom{f_2} \in \C[\hom{f_4}],\;
P(f_2,f_4) \in \C[f_4] \text{ or } \hom{f_3} \in \C[\hom{f_4}],$$
then $E \circ F \in \A$.
\end{lem}

\begin{proof}
(i) Assume $Q(f_3,f_4)=Q(f_4)\in\C[f_4]$ (see Figure \ref{fig:Q}).

 Then one checks: $E'':=E\circ E'\circ E^{-1} \in \E^1_3$.
Using Induction Hypothesis \ref{IH} we get $E\circ E' \circ F \in \A$ and, applying it once again in order to get $E \circ F\in \A$, we are left to prove that $\deg  E\circ E' \circ F<\deg F$. For this, we remark that $F$ and $E\circ F$ resp. $E'\circ F$ and $E''\circ E\circ F$ have the same 3rd component: $f_3$ resp. $f_3':=f_3+f_4P(f_2,f_4)$. By Lemma \ref{lem:basic}, the assumption $\deg E'\circ F<\deg F$ translates in $\deg f_3'<\deg f_3$ but this in turn translates   in $\deg E''\circ E\circ F<\deg E\circ F$ and we are done since, by assumption,  $\deg E\circ F\leq\deg F$.
  \\

\begin{figure}[h]
$$\xymatrix{  
 F = \mat{f_1}{f_2}{f_3}{f_4} \ar[d]_<(.3){E'} \ar[dr]^>(.8)E \\
 E'\circ F  \in \A_{<F} \ar[dr]_E & E\circ F  \ar[d]^<(.2){E''}\\
 &   \in \A_{<F}
}
\qquad
\xymatrix{  
F = \mat{f_1}{f_2}{f_3}{f_4} \ar[d]_<(.3){E'} \ar[dr]^>(.8){\tilde E} \ar[rr]^E
&& E \circ F\\
E'\circ F \in \A_{<F} \ar[dr]_{\tilde E} & \tilde E\circ F \in \A_{<F} \ar[d]^<(.2){\tilde{E}''} \ar[ur]^<(.2){E''}\\
 & \in \A_{<F}
}
$$
\caption{Cases (i) $Q \in \C[f_4]$, and (ii) $\hom{f_2} \in \C[\hom{f_4}]$ in Lemma \ref{lem:conditions}.}\label{fig:Q}
\end{figure}

(ii) Assume $\hom{f_2} \in \C[\hom{f_4}]$ (see Figure \ref{fig:Q}).

Then there exists $\tilde Q(f_4)\in\C[f_4]$ such that $\bdeg f_2+f_4\tilde Q(f_4)<\bdeg f_2$.
We take $\tilde E =\mat{x_1+x_3\tilde Q(x_4)}{x_2+x_4\tilde Q(x_4)}{x_3}{x_4}$, and we have  $\tilde E\circ F\in\A$ by case (i). 
Thus $\tilde E\circ F\in\A_{<F}$. 
We conclude using Induction Hypothesis \ref{IH} on $E''\circ (E\circ F)$ where $E'' =E\circ \tilde E^{-1}\in\E^{12}$. \\

(iii) Assume $P(f_2,f_4)=P(f_4) \in \C[f_4]$ (see Figure \ref{fig:P}).

Consider $E'' := E'\circ E\circ {E'}^{-1} \in \Et$.
First by Induction Hypothesis \ref{IH} we have $E''\circ E' \circ F \in \A$. 
If we can prove $\deg E''\circ E' \circ F < \deg F$ then we can use the Induction Hypothesis again to obtain $E \circ F \in \A$. 
But this is done as follows, using a similar argument as in case (i).
We note that $f_2$ is the second coordinate of both $F$ and $E' \circ F$, and $f_2 + f_4Q(f_3,f_4)$ is the second coordinate of both $E\circ F$ and $E''\circ E' \circ F$.
By Lemma \ref{lem:basic} the assumption $\deg E \circ F \le \deg F$ is equivalent to $\deg f_2 + f_4Q(f_3,f_4) \le \deg f_2$, which in turn in equivalent to $\deg E''\circ E' \circ F \le \deg E' \circ F$. This gives the result, since $ \deg E' \circ F < \deg F$.\\

(iv) Finally assume $\hom{f_3} \in \C[\hom{f_4}]$ (see Figure \ref{fig:P}).

There exists $\tilde P(f_4)\in\C[f_4]$ such that $\bdeg f_3+f_4\tilde P(f_4)<\bdeg f_3$.
We take $\tilde E =\mat{x_1+x_2\tilde P(x_4)}{x_2}{x_3+x_4\tilde P(x_4)}{x_4}$, and we have  $\tilde E\circ F\in\A$ by Case \ref{case:easy}. 
Thus $\tilde E\circ F\in\A_{<F}$. 
We conclude using case (iii).

\begin{figure}[h]
$$\xymatrix{  
 F = \mat{f_1}{f_2}{f_3}{f_4} \ar[d]_{E'} \ar[r]^>(.8){E} 
& E\circ F \ar[d]^{E'} \\
 E'\circ F \in \A_{<F} \ar[r]_>(.8){E''}  & \in \A_{<F}
}
\qquad
\xymatrix{  
& F = \mat{f_1}{f_2}{f_3}{f_4} \ar[dl]_<(.3){E'} \ar[d]^{\tilde E} \ar[r]^>(.8)E
& E \circ F \ar[d]^{\tilde E}\\
E'\circ F \in \A_{<F} \ar[r]_{\tilde E \circ {E'}^{-1}} & \tilde E\circ F \in \A_{<F} \ar[r]_>(.8){\tilde{E}''}  & \in \A_{<F}
}
$$
\caption{Cases (iii) $P \in \C[f_4]$, and (iv) $\hom{f_3} \in \C[\hom{f_4}]$ in Lemma \ref{lem:conditions}.}\label{fig:P}
\end{figure}

\end{proof}

%

\section{Examples of wild automorphisms} \label{sec:wild}

\subsection{The case of $\SL_2(\C)$}

We consider automorphisms $\sigma_n$ of the form $\exp ((x_1+x_4)^n \partial)$, where $\partial$ is the locally nilpotent derivation that we introduced in \S \ref{intro:tame}.
We have
$$\sigma_n = 
\begin{pmatrix} x_1 -x_2(x_1+x_4)^n & x_2 \\ 
x_3 + (x_1-x_4)(x_1+x_4)^n -x_2(x_1+x_4)^{2n} & x_4 + x_2(x_1+x_4)^n \end{pmatrix}$$
The automorphism $\sigma$ of the introduction corresponds to $n = 1$.

Now assume that $\sigma_n$ admits an elementary reduction $E \circ \sigma_n$. 
Since the degree of the second coordinate cannot decrease, by Lemma \ref{lem:basic} we see that $E \in \El$ or $\Eb$.

Since both cases are symmetrical, we only consider the former one. 
By Lemma \ref{lem:inA} we can assume that
$$E = \bmat{x_1+x_2P(x_2,x_4)}{x_2}{x_3+x_4P(x_2,x_4)}{x_4}.$$
 
Computing the leading parts of the coordinates of $\sigma_n$, which are $\mat{-x_2x_4^n}{x_2}{-x_2x_4^{2n}}{x_2x_4^n}$, we see that  
$$-x_2x_4^{n} + x_2P(x_2,x_2x_4^n) = 0.$$
This implies that $x_4^n$ is in $\C[x_2,x_2x_4^n]$, which is contradictory. 
Hence $\sigma_n$ does not admit an elementary reduction, and by Theorem \ref{thm:main} we conclude that $\sigma_n$ is not a tame automorphism.

\subsection{The case of $\PSL_2(\C)$}

One can adapt the discussion of this note to the case of the automorphism group of the complement of a smooth quadric surface in $\cpt$; in other words to the context of $\Aut(\PSL_2(\C))$.
Consider the double cover 
$$\pi\colon \SL_2(\C) \to \PSL_2(\C) = \SL_2(\C)/\langle -\id \rangle.$$
Clearly if $f \in \Aut(\SL_2(\C))$ commutes with $-\id$ then it induces an automorphism $F \in \Aut(\SL_2(\C))$ such that $\pi \circ f = F \circ \pi$.
The following observation was pointed to us by J\'er\'emy Blanc:

\begin{lem}
Let $F$ be an automorphim of $\PSL_2(\C)$.
Then there exists $f \in \Aut(\SL_2(\C))$ such that $\pi \circ f = F \circ \pi$.
\end{lem}

\begin{proof}
An automorphism $F$ of $\PSL_2(\C) = \cpt \smallsetminus \{ q = 0 \}$ is given by four homogeneous polynomial $f_i$ of the same degree. 
If $F$ is linear, the surface $ \{ q = 0 \}$ is preserved by $F$; and in the non-linear case the locus $q \circ F = 0$ corresponds to divisors in $\cpt$ contracted by $F$, which must be supported on $q = 0$. 
In both cases we obtain that the polynomial $q\circ F \in \C[x_1,\dots,x_4]$ is a power of $q$, up to a constant. 
Multiplying the $f_i$ by a constant, we can thus assume $q \circ F = q^k$.
The same remark applies to the four homogeneous polynomial $g_i$ associated with $F^{-1}$. 
Thus $f = (f_1, \dots, f_4)$ anf $g = (g_1, \dots, g_4)$, viewed now as endomorphisms of $\Af^4$, preserve the level $q = 1$, or in other words $\SL_2(\C) \subset \Af^4$.
We have
$$ f \circ g = (H x_1, \dots, H x_4)$$
where $H = c q^n$ is a power of $q$ up to a constant $c$ satisfying $q(cx_1, \dots, cx_4) = q(x_1, \dots, x_4)$. Hence $c = \pm 1$. 
Multiplying if necessary the $f_i$ by a square root of $-1$ (but not touching the $g_i$), we can remove the sign and obtain an automorphism $f$ on $\SL_2(\C)$, with inverse $g$, and which by definition satisfies $\pi \circ f = F \circ \pi$.
\end{proof}
 
Note that the automorphism $f$ given by the proposition commutes with $-\id$ and is uniquely defined up to a sign. 

Now we can define for instance the group $\Eb \subset \PSL_2(\C)$ as the group of automorphisms $F$ such that there exists $f \in \Eb \subset \SL_2(\C)$ satisfying $\pi \circ f = F \circ \pi$. 
Explicitely these are automorphisms of the form
\begin{equation*}
\begin{pmatrix} x_1/a & x_2/b \\ 
bx_3 + bx_1\frac{h(x_1,x_2)}{(x_1x_4 - x_2x_3)^n} & ax_4 + ax_2\frac{h(x_1,x_2)}{(x_1x_4 - x_2x_3)^n} \end{pmatrix}
\end{equation*}
where $h(x_1,x_2)$ is a homogeneous polynomial of (ordinary) degree $2n$.

Other types of elementary automorphisms are defined in a similar way.
Thus we obtain a tame group and deduce from the discussion above that for instance

$$\begin{pmatrix} x_1 -x_2\frac{(x_1+x_4)^2}{x_1x_4 - x_2x_3} & x_2 \\ 
x_3 + (x_1-x_4)\frac{(x_1+x_4)^2}{x_1x_4 - x_2x_3} -x_2\frac{(x_1+x_4)^4}{(x_1x_4 - x_2x_3)^2} & x_4 + x_2\frac{(x_1+x_4)^2}{x_1x_4 - x_2x_3} \end{pmatrix}$$
which is the push-forward by $\pi$ of the automorphisms $\sigma_2$ of the previous paragraph, is not tame.


\bibliographystyle{alpha}
\bibliography{biblio}

\end{document}